\documentclass[12pt,a4paper,english]{smfart}
\usepackage[english]{babel}
\usepackage{smfthm,mathabx}
\usepackage{amssymb}
\usepackage{amsmath}
\usepackage{amsthm}
\usepackage{amsfonts}
\usepackage{graphicx}
\usepackage{enumerate}

\usepackage{appendix}

\usepackage{euscript,mathrsfs}
\usepackage[utf8]{inputenc} 
\usepackage{longtable}
\usepackage{dsfont}

\usepackage[OT2,T1]{fontenc}

\usepackage[all,cmtip]{xy}
\usepackage{alltt}

\usepackage{calrsfs}

\xyoption{all}

\usepackage{float}

\usepackage{fullpage}

\usepackage{color}

\newtheorem{Theorem}{Theorem}

\newtheorem{Remark}{Remark}

\author[Oussama Hamza]{Oussama Hamza}
\address{Ecole Normale Supérieure de Lyon \\ Université de Lyon, 15 parvis René Descartes \\69342 Lyon Cedex 07 \\ France  }
\email{oussama.hamza@ens-lyon.fr}

\author{Christian Maire}
\address{FEMTO-ST Institute \\ Universit\'e Bourgogne Franche-Comt\'e, 15B Avenue des Montboucons  \\25030 Besan\c con Cedex\\ France }
\email{christian.maire@univ-fcomte.fr}

\title{A note on asymptotically good extensions in  which infinitely many primes split completely}

\thanks{The authors thank Farshid Hajir for useful comments, and Philippe Lebacque for his interest in this work. 
CM was partially supported by the ANR project FLAIR (ANR-17-CE40-0012), and by the EIPHI Graduate School (ANR-17-EURE-0002)}

\newcommand{\Q}{\mathbb{Q}}
\newcommand{\F}{\mathbb{F}}
\newcommand{\Z}{\mathbb{Z}}

\def\Ss{\mathbb{S}}

\def\N{{\rm N}}

\def\G{{\rm G}}

\def\p{{\mathfrak p}}

\def\Disc{{\rm Disc}}
\def\log{{\rm log}}
\def\Cl{{\rm Cl}}
\def\mod{{\rm mod}}

\def\grad{{\rm Grad}}
\def\rd{{\rm rd}}
\def\Gal{{\rm Gal}}
\def\cd{{\rm cd}}
\def\Nor{{\rm Nor}}
\def\Ee{{\rm E}}
\def\Cc{{\rm C}}

\def\O{{\mathcal O}}

\def\E{{\mathcal E}}

\def\AA{{\rm A}}
\def\FF{{\rm F}}
\def\R{{\rm R}}

\def\Ff{{\mathcal F}}

\def\BB{{\rm B}}
\def\II{{\rm I}}

\def\L{{\rm L}}
\def\K{{\rm K}}

\def\Zz{Z}

\def\V{{\rm V}}

\def\W{{\rm W}}

\def\p{{\mathfrak p}}

\def\Ss{{\mathcal S}}

\date{\today}

\begin{document}

\maketitle

%%%%%%%%%%%%%%%%%%%%%%%%%%%%%%%%%%%%%%%%%%%%%%%%%%%%%%%%%%%%%%%%%%

\begin{abstract} Let $p$ be a prime number, and let $\K$ be a number field. For $p=2$, assume moreover $\K$ totally imaginary. In this note we prove the existence of asymptotically good extensions $\L/\K$ of cohomological dimension~$2$ in which infinitely many primes split completely. Our result is inspired by a recent work of Hajir, Maire, and Ramakrishna \cite{HMR}. 
\end{abstract}

%%%%%%%%%%%%%%%%%%%%%%%%%%%%%%%%%%%%%%%%%%%%%%%%%%%%%%%%%%%%%%%%%%

Let $\K$ be a number field, and let $\L/\K$ be an infinite unramified extension.   Denote by $\Ss_{\L/\K}$ the set of prime ideals of $\K$ that split completely in $\L/\K$. In \cite{Ihara} Ihara proved that  $\displaystyle{\sum_{\p \in \Ss_{\L/\K}} \frac{ \log \N(\p)}{\N(\p)} < \infty}$,
and raised the following interesting question: are there  $\L/\K$ for which   $\Ss_{\L/\K}$ is infinite ? This question was recently answered in the positive by Hajir, Maire, and Ramakrishna in \cite{HMR}.
In fact, infinite unramified extensions  $\L/\K$ are some special cases of infinite extensions for which  the  root discriminants $\rd_\FF:=|\Disc_\FF|^{1/[\FF:\Q]}$ are bounded, where the number fields~$\FF$ vary in $\L/\K$, and $\Disc_\FF$ is the discriminant of~$\FF$. Such extensions are called {\it asymptotically good}, and it is now well-known  that in such extensions the inequality of Ihara involving $\Ss_{\L/\K}$ still holds (see for example \cite{TV}, or \cite{Lebacque} for the study of such extensions).

\medskip

Pro-$p$ extensions of number fields with restricted ramification allow us to exhibit asymptotically good extensions. Let $p$ be a prime number, and let $S$ be a finite set of prime ideals of $\K$ coprime to $p$ (more precisely each $\p \in S$ is such that $|\O_\K/\p| \equiv 1({\rm mod} \ p)$); the set~$S$ is called {\it tame}. Let $\K_S$ the maximal pro-$p$ extension of $\K$ unramified outside~$S$, put $\G_S=\Gal(\K_S/\K)$. In $\K_S/\K$ the root discriminants are bounded by some constant depending on the discriminant of $\K$ and the norm of the places of $S$ (see for example \cite[Lemma 5]{Hajir-Maire-0}). Moreover thanks to Golod-Shafarevich criterion, it is well-known that $\K_S/\K$ is infinite when $|S|$ is large  as compared to $[\K:\Q]$ (see for example \cite[Chapter  X, \S 10, Theorem 10.10.1]{NSW}), and then asymptotically good. {\it E.g.} for $p>2$, $\Q_S/\Q$ is infinite when $|S|\geq 4$. 
In \cite{HMR} the authors showed that when $S$ is large, there exist infinite subextension  $\L/\K$ of $\K_S/\K$ for which the set $\Ss_{\L/\K}$ is infinite. But they  give no information about the structure of $\Gal(\L/\K)$.
Here we prove:

\begin{Theorem} \label{theo_main} Let $p$ be a prime number, and let $\K$ be a number field. For $p=2$ assume~$\K$ totally imaginary. Let $T$ and $S_0$ be two disjoint finite sets of prime ideals  of~$\K$ where~$S_0$ is tame. Then for infinitely many finite sets~$S$ of tame prime ideals of $\K$ containing $S_0$ there exist an infinite pro-$p$ extension $\L/\K$ in $\K_S/\K$ such that
 \begin{enumerate}[\quad (i)]
  \item the set $\Ss_{\L/\K}$ of places that split completely in $\L/\K$ contains $T$;
  \item the set $\Ss_{\L/\K}$  is infinite;
  \item the pro-$p$ group $\G=\Gal(\L/\K)$ is of cohomological dimension $2$;
  \item the minimal number of relations  of $\G$ is infinite, {\it i.e.} $\dim H^2(\G,\F_p)=\infty$;
  \item for each $\p\in S$, the local extension $\L_\p/\K_\p$ is maximal, {\it i.e.} isomorphic to $\Z_p  \rtimes \Z_p$;
  \item the Poincaré series of the algebra $\F_p\ldbrack \G \rdbrack$, endowed with the graduation from the  ideal of augmentation, is equal to 
  $\displaystyle{\big({1-dt+rt^2+ t^3\sum_{n\geq 0} t^{n}}\big)^{-1}}$, where $d=\dim \G_S$, and where $r$ is explicit, depending on  $\K,S,T$.
 \end{enumerate} 
\end{Theorem}

 \begin{Remark}
 We will see that the pro-$p$ group $\G$ of  Theorem \ref{theo_main} is {\it mild} in the terminology of Anick \cite{Anick1}. See also  Labute \cite{Labute} for  arithmetic contexts.
 \end{Remark}

The proof uses various tools.

The first one is  the strategy developed initially by Labute \cite{Labute}, then by Labute-Min\'a\v{c} \cite{Labute-Minac}, Schmidt \cite{Schmidt},  Forr\'e \cite{Forre} etc. for studying   the  cohomological dimension of~ a pro-$p$ group $\G$,  through the  notion of strongly free sets introduced by Anick \cite{Anick}. By following the approach of Forr\'e \cite{Forre},  we refine  this idea when the minimal number of relations of $\G$ is infinite.

 This key idea is associated to a result of Schmidt \cite{Schmidt} that shows  that the pro-$p$ group $\G_S$ is of cohomological dimension $2$ for some well-chosen~$S$; the proof of Schmidt involves the cup-product $H^1(\G_S,\F_p) \cup H^1(\G_S,\F_p)$. Here we use the translation of this cup-product in the   polynomial algebra, due to Forr\'e.  In particular,  this allows us to choose infinitely many Frobenius in $\G_S$ such that the family of the highest terms of these plus the highest terms of the relations of $\G_S$, is  combinatorially free (see \S \ref{section_dimcoh_polynomial} and Definition \ref{def_cf}). 

 We conclude  by cutting the tower $\K_S/\K$ by all these  Frobenius: this is the strategy of~\cite{HMR}.

\medskip

This note contains two sections. In \S 1 we recall the results we need regarding  pro-$p$ groups, graded algebras, and arithmetic of pro-$p$ extensions with restricted ramification. In \S 2 we start with an example when $\K=\Q$, and prove the main result.

%%%%%%%%%%%%%%%%%%%%%%%%%%%%%%%%%%%%%%%%%%%%%%%%%%%%%%%%%%%%%%%%ùùù

%%%%%%%%%%%%%%%%%%%%%%%%%%%%%%%%%%%%%%%%%%%%%%%%%%%%%%%%%%%%%%%%ùùù
%%%%%%%%%%%%%%%%%%%%%%%%%%%%%%%%%%%%%%%%%%%%%%%%%%%%%%%%%%%%%%%%ùùù

\

{\bf Notations.}

 Let $p$ be a prime number.

$\bullet$ If ${\rm V}$ is a $\F_p$-vector space we denote by $\dim {\rm V}$ its dimension over $\F_p$. 

$\bullet$  For a pro-$p$ group $\G$, we denote by $H^i(\G)$ the cohomology group $H^i(\G,\F_p)$.
The $p$-rank of $\G$, which is equal to $\dim H^1(\G)$, is noted $d_p \G$. 
%If $A$ and $B$ are two subgroups of $G$, we denote by $AB$, the closed subgroup generated by $\{ab; a\in A, b\in B\}$, and $A^p$, the closed subgroup generated by $\{a^p; a\in A\}$.

%%%%%%%%%%%%%%%%%%%%%%%%%%%%%%%%%%%%%%%%%%%%%%%%%%%%%%%%%%%%%%%%ùùù

%%%%%%%%%%%%%%%%%%%%%%%%%%%%%%%%%%%%%%%%%%%%%%%%%%%%%%%%%%%%%%%%ùùù
%%%%%%%%%%%%%%%%%%%%%%%%%%%%%%%%%%%%%%%%%%%%%%%%%%%%%%%%%%%%%%%%ùùù

\section{The results we need} \label{section_gradued_algebra}
 
\subsection{On pro-$p$ groups}

For this section we refer to \cite{Brumer}, \cite[Chapters 5,6 and 7]{Koch}, and \cite{Forre}.
Take a prime number $p$.

\subsubsection{Minimal presentation and cohomological dimension} \label{section_presentation}

Let  $\G$ be a  pro-$p$ group of finite rank $d$, and let   
$1\to \R \to \FF \to \G \to 1$ be a minimal presentation of $\G$ by a free pro-$p$ group $\FF$.
Let $\Ff:=\{\rho_i\}_{i\in I}$ be an $\F_p$-basis of $\R/\R^p[\FF,\R]$; observe that $I$ is not necessarily finite.
The algebra $\Lambda_\G:=\F_p\ldbrack \G\rdbrack$ acts on $\R/\R^p[\R,\R]$, and by Nakayama's lemma 
 the $\rho_i$'s generate topologically $\R/\R^p[\R,\R]$ as $\Lambda_\G$-module (see for example \cite[Corollary 1.5]{Brumer}).

\medskip

Let us recall the definition of the cohomological dimension $\cd(\G)$ of  $\G$: it is the smallest integer $n$ (eventually $n=\infty$) such that $H^i(\G)=0$ for every $i\geq n+1$. 

\medskip

\begin{theo} \label{theo_dimcoh}
The following assertions are equivalent:
\begin{enumerate}[\quad (i)]
    \item $\cd(\G) \leq 2$;
    \item $\R/\R^p[\R,\R]$ is a free compact $\Lambda_\G$-module;
    \item $\displaystyle{\R/\R^p[\R,\R]  \simeq \prod_I \Lambda_\G}$.
\end{enumerate}
Moreover, $\dim H^2(\G)=|I|$. %In particular $\cd(\G)=2$ if and only if the set $I$ is not empty.
\end{theo}

\begin{proof} See \cite[Corollary 5.3]{Brumer} or \cite[Chapter 7, \S 7.3, Theorem 7.7]{Koch}.
\end{proof}

We are  going to translate conditions of Theorem \ref{theo_dimcoh} in the algebra $\F_p^{nc}\ldbrack X_1,\cdots, X_d\rdbrack$.

%%%%%%%%%%%%%%%%%%%%%%%%%%%%%%%%%%%%%%%%%%%%%%%%%%%%%%%%%%%%%%%%ùùù

%%%%%%%%%%%%%%%%%%%%%%%%%%%%%%%%%%%%%%%%%%%%%%%%%%%%%%%%%%%%%%%%ùùù
%%%%%%%%%%%%%%%%%%%%%%%%%%%%%%%%%%%%%%%%%%%%%%%%%%%%%%%%%%%%%%%%ùùù

\subsubsection{Filtred and graded algebras} \label{section_graded_algebra2} The results of this section can be found in~\cite{Anick}.

$\bullet$ Let $$\Ee= \F_p^{nc} \ldbrack X_1, \cdots, X_d\rdbrack$$ be the algebra of noncommutative series in $X_1,\cdots, X_d$ with coefficients in $\F_p$.
We consider now noncommutative multi-indices $\alpha=(\alpha_1,\cdots, \alpha_n)$, with $\alpha_i \in \{1,\cdots, d\}$, and we denote by  $X_\alpha$ the monomial element of the form $X_\alpha=X_{\alpha_1}\cdots X_{\alpha_n}$.  We endow each $X_i$ with the degree $1$;  the degree $\deg(X_\alpha)$ of $X_\alpha$ is $|\alpha|$. 

For $\displaystyle{\Zz=\sum_{\alpha} a_\alpha X_\alpha}$,  the quantity  $\omega(\Zz)=\min_{a_\alpha \neq 0}\{\deg(X_\alpha) \}$ is the valuation of $\Zz$, with the convention that $\omega(0)=\infty$.
For $n \geq 0$, put $\Ee_n=\{ \Zz\in \Ee, \omega(\Zz) \geq n\}$. Observe that $\Ee_1$ is the augmentation ideal of $\Ee$:  this is the two-sided ideal of $\Ee$ topologically generated by the $X_i$'s. The algebra $\Ee$   is filtered by the $\Ee_n$'s and its
graded algebra $\grad(\Ee)$ is then:
$$\grad(\Ee)=\bigoplus_{n\in \Z_{\geq 0}}{\Ee_n/\Ee_{n+1}}\simeq \F_p^{nc}\lbrack X_1,\dots,X_d\rbrack.$$
In other words $\grad(\Ee)$ is isomorphic to the noncommutative polynomial algebra $\AA:=\F_p^{nc}\lbrack X_1,\dots,X_d\rbrack$, where each $X_i$ is endowed with the formal degree $1$.
Let $\AA_n=\{z\in\AA,  \omega(z) \geq n\}$ be the gradation of $\AA$; observe that $\AA_1$ is the augmentation ideal of $\AA$.

\medskip
 
 $\bullet$ Let $X_\alpha ,X_{\alpha'}$ be two monomials (viewed in $\Ee$ or in $\AA$).  The element $X_\alpha$ is a {\it submonomial} of $X_{\alpha'}$, if $X_{\alpha'}=X_\beta X_\alpha X_{\beta'}$, with $X_\beta , X_{\beta'} $ two monomials of $\AA$.

\begin{defi} \label{def_cf}
A family $\Ff=\{X_{\alpha^{(i)}}\}_{i\in I}$ of monomials of $\AA$ is combinatorially free if for all $i\neq j$:  
\begin{enumerate}[\quad (i)]
    \item $X_{\alpha^{(i)}}$ is not a submonomial of $X_{\alpha^{(j)}}$,
    \item  if  $X_{\alpha^{(i)}}=X_\alpha X_\beta$ and $X_{\alpha^{(j)}}=X_{\alpha'}X_{\beta'}$, then $X_\alpha\neq X_{\alpha'}$, with $X_\alpha, X_\beta,X_{\alpha'},X_{\beta'}$ non-trivial monomials, {\it i.e.} $\neq 1$.
\end{enumerate}
\end{defi}

The monomials  may be endowed with a total order $<$ as follows. 

First let us consider the natural ordering $<'$ defined by: $X_1<'X_2<'\cdots <' X_d$.

Let $X_\alpha$ and $X_\beta$ two monomials, we say that $X_\alpha > X_\beta$, if $\omega(X_\alpha) < \omega(X_\beta)$; if $X_\alpha$ and $X_\beta$ have the same valuation, we use the lexicographic order induced by $<'$.

 \medskip
 
Now, let $Z=\sum_\alpha a_\alpha X_\alpha$ be a nonzero element of $\Ee$, with $a_\alpha\in \F_p$. 
Then $\widehat{Z}:=\max \{X_\alpha, \ a_\alpha\neq 0\}$ is the {\it highest term} respecting the  order $<$. 

\medskip

$\bullet$ 
Let $\Ff:=\{Z_i\}_{i\in I}$ be a locally finite graded subset of $\AA_1$ generating $\Cc$ as two-sided $\AA$-ideal: $\Cc=\AA \Ff \AA$. Observe that $I$ is countable.
Let  $\BB:=\AA/\Cc$ be the quotient endowed with the quotient gradation; we denote by $P_\BB(t)=\sum_{n\in \Z_{\geq 0}}{\dim (\BB_n/\BB_{n+1}) \cdot t^n}$ the Poincar\'e series of~$\BB$. Observe that   the family  $\Ff$  generates the $\BB$-module $\Cc/\Cc \AA_1$.

\begin{theo}[Anick]\label{USE}
Let $\Ff=\{Z_i\}_{i\in I}$ be a locally finite graded subset of $\AA_1$, and let $\Cc$ be a two-sided ideal of $\AA$ generated  by  the $Z_i$'s; put $\BB=\AA/\Cc$. For each $i$, let $X_{\alpha^{(i)}}:=\widehat{Z_i}$  be the highest term of~$Z_i$.
If the family $\{X_{\alpha^{(i)}}\}_{i\in I}$ is combinatorially  free, then
\begin{enumerate}[\quad (i)]
 \item $\Cc/\Cc \AA_1$ is a free $\BB$-module over the $Z_i$'s, and 
 \item $\displaystyle{P_\BB(t)=\big({1-dt+\sum_{i\in I} t^{n_i}}\big)^{-1}}$, where $n_i=\omega(Z_i)=\omega(X_{\alpha^{(i)}})$.
 \end{enumerate}
\end{theo}

\begin{proof} See \cite[Theorems 2.6 and 3.2]{Anick}.
\end{proof}

If $\Cc/\Cc \AA_1$ is a free $\BB$-module over the $Z_i$'s, we say that the family $\Ff=\{Z_i\}_{i \in I}$ is {\it strongly free} (see \cite{Anick}).

\begin{exem} \label{exemple_inert_set} Take $d=5$,  and the lexicographic ordering $X_1<X_2< \cdots < X_5$.
Let $a_n \geq 1$ be an increasing sequence, $n \geq 1$, and consider the family $\Ff=\{X_5X_3,X_4X_2,X_4X_3,X_5X_2,X_5X_1,X_5X^{a_n}_4X_1, n\geq 1\}$. Put $\BB=\AA/\AA \Ff \AA$. Then $\Ff$ is combinatorially free, and 
 $\displaystyle{P_{\BB}(t)=\big({1-5t+t^2\sum_{n\geq 1} t^{a_n}}\big)^{-1}}$.
\end{exem}

%%%%%%%%%%%%%%%%%%%%%%%%%%%%%%%%%%%%%%%%%%%%%%%%%%%%%%%%%%%%%%%%ùùù

%%%%%%%%%%%%%%%%%%%%%%%%%%%%%%%%%%%%%%%%%%%%%%%%%%%%%%%%%%%%%%%%ùùù

%%%%%%%%%%%%%%%%%%%%%%%%%%%%%%%%%%%%%%%%%%%%%%%%%%%%%%%%%%%%%%%%ùùù

\subsubsection{Pro-$p$ groups of cohomological dimension $\leq 2$ and polynomial algebra} \label{section_dimcoh_polynomial}

Let us conserve the notations of \S \ref{section_presentation}.

Let $\FF$ be a free pro-$p$ group on $d$ generators $x_1,\cdots, x_d$. Let $\Lambda_\FF:=\F_p\ldbrack \FF \rdbrack$ be the complete algebra associated to $\FF$. 
Recall that $\F_p\ldbrack \FF \rdbrack$ is isomorphic to the Magnus algebra $\Ee=\F_p^{nc}\ldbrack X_1,\cdots, X_d \rdbrack$; this isomorphism $\varphi$ is given by $x_i \mapsto X_i+1$ (see for example \cite[Chapter 7, \S 7.6, Theorem 7.16]{Koch}).

Let us endow $\Ee$ with the filtration and the ordering of \S \ref{section_graded_algebra2}.
The filtered isomorphism $\varphi:\Lambda_\FF \stackrel{\simeq}{\to} \Ee$ allows us to endow $\Lambda_\FF$ with the valuation $\omega_\FF$ defined as follows: $\omega_\FF(z)=\omega(\varphi(z))$.
Observe that   $\Ee_1\simeq \II_\FF:\ker(\Lambda_\FF \rightarrow \F_p)$, that is $\Ee_1$ is isomorphic to  the augmentation ideal of $\Lambda_\FF$.

Take $x\in \FF$, $x\neq 1$. Then the degree $\deg(x)$ of $x$ is defined as $\deg(x):=\omega_\FF(x-1)=\omega(\varphi(x-1))$. We denote by $\widehat{x}$ the highest term of $\varphi(x-1) \in \Ee$. Hence $\widehat{x}$ is a monomial.

\begin{exem} \label{exemple_clef} Take $d\geq 3$ with the lexicographic ordering $X_1<X_2<X_3< \cdots<X_d$.
\begin{enumerate}[\quad (i)]
 \item  The highest term of $[x_1,[x_2^{p^n},x_3]]$ is $X_3X_2^{p^n}X_1$.
 \item Given $x,y\in \FF$, let us write $f_x(y)=[x,y] \in \FF$. Then the highest term of $f_{x_1}\circ f_{x_2}^{\circ^n}(x_3)$ is $X_3X_2^nX_1$.
\end{enumerate}
\end{exem}

Let   $\G$ be a pro-$p$ group of $p$-rank $d$, and let 
$1\to \R \to \FF \to  \G \to 1$ be a minimal presentation of $\G$ by  $\FF$; this induces a filtered morphism $\theta : \Lambda_\FF \to \Lambda_\G$.
We  now endow $\Lambda_\G$ with the induced valuation $\omega_\G$ of $\omega_\FF$ as follows: for $z\in \Lambda_\G$, let us define 
$$\omega_\G(z)=\max\{\omega_\FF(z'), z' \in \Lambda_\FF,  \theta(z')=z\}.$$
Put $\Ee_{\G,n}=\{z\in \Lambda_\G, \omega_\G(z)\geq n\}$, the filtration of $\Lambda_\G$. Then $\grad(\Lambda_\G)=\bigoplus_n \Ee_{\G,n}/\Ee_{\G,n+1}$ is the graded algebra  of $\F_p\ldbrack \G \rdbrack$ respecting the quotient gradation with  $\displaystyle{P_\G(t)=\sum_{n\geq 0} \dim \Ee_{\G,n}/\Ee_{\G,n+1} \cdot t^n}$ as Poincaré series. 

For $n\geq 1$, put $\FF_n:=\{x\in \FF, \varphi(x-1) \in \Ee_n\}$, and $\G_n=\FF_n\R/\R$. The sequences $(\FF_n)$ and $(\G_n)$ are the Zassenhaus filtrations of $\FF$ and $\G$.
The filtration $(\Ee_{\G,n})$ corresponds also to the filtration coming from the augmentation ideal of $\Lambda_\G$ (see for example \cite[Appendice A.3, Th\'eor\`eme 3.5]{Lazard}).

\begin{theo} \label{EQ}
Let   $\Ff=\{\rho_i\}_{i\in I}$ be a family of generators $\R/\R^p[\R,\R]$.
For each $\in I$, let $X_{\alpha^{(i)}}=\{\widehat{\rho_i}\}_{ i \in I} \in \AA$ be the highest term of $\rho_i$.
If  $\{X_{\alpha^{(i)}}\}_{ i \in I}$ is combinatorially free, then
\begin{enumerate}[\quad (i)]
\item $\displaystyle{\R/\R^p[\R,\R] \simeq \prod_{i\in I} \Lambda_\G}$, and $\cd(\G)\leq 2$;
\item $\displaystyle{P_\G(t)=\big({1-dt +\sum_{i\in I} t^{n_i}}\big)^{-1}},$ where $d=d_p \G$, and  $n_i=\deg(\rho_i)=\omega(X_{\alpha^{(i)}})$.
\end{enumerate}
\end{theo}

\begin{proof} When the set of indexes $I$ is finite, this version can be found in \cite{Forre}. We show here that the result also holds when  $I$  is infinite. First, observe  that as $\{X_{\alpha^{(i)}}\}_{ i \in I}$ is combinatorially free then $I$ is countable infinite.

\medskip

For $i\in I$, put $Y_i=\varphi(\rho_i-1) \in \Ee_1$; $n_i=\omega(Y_i)$.
Let $\II(\R) \subset \Ee_1$ be the closed two-sided ideal of $\Ee_1$ topologically generated  by the  $Y_i$'s, $i\in I$; one has $\ker(\theta) \simeq \II(\R)$ (see for example \cite[Chapter 7, \S 7.6, Theorem 7.17]{Koch}).
 Let us recall now the  topological $\G$-isomorphism between $\R/\R^p[\R,\R]$ and $\II(\R)/\II(\R) \Ee_1$ (see for example \cite[Proposition 4.3]{Forre}). We want to  some informations on the $\G$-module $\R/\R^p[\R,\R]$, and then on $\II(\R)/\II(\R) \Ee_1$.
 
\medskip

For $i\in I$, let $Z_i \in \AA$ be the initial form of $Y_i \in \Ee_1$ defined as follows: let us write $Y_i=Z_{i,n_i}+Z_{i,n_i+1}+\cdots$, where $n_i=\omega(Y_i)$ and where $Z_{i,j}$ are homogeneous polynomial of degree~$j$ (eventually $Z_{i,j}=0$); then put $Z_i=Z_{i,n_i}$.
Observe that $\widehat{\rho_i}=\widehat{Y_i}=\widehat{Z_i}$.

\medskip

Let $\Cc$ be the closed ideal of $\AA=\F_p^{nc}\lbrack X_1,\cdots, X_d\rbrack$ generated by the family $\{Z_i\}_{i\in I}$.
As the family $\{\widehat{\rho_i}\}_{i\in I}$ is combinatorially free then by Theorem \ref{USE} the family $\{Z_i\}_{i\in I}$ is strongly free. 
Put $\BB=\AA/\Cc$. 

\begin{prop} \label{lemm_clef} One has $\Cc=\grad(\II(\R)) \subset \AA$. In particular, as graded $\AA$-modules, one gets
  $\grad(\Lambda_\G) \simeq \BB$, and $$\displaystyle{\grad(\II(\R)/\II(\R)\Ee_1) \simeq \Cc/\Cc\AA_1\simeq \bigoplus_{i \in I} \BB Z_i\simeq \bigoplus_{i \in I} \BB[n_i]},$$ where $\BB[n_i]$ means $\BB$ as $\AA$-module  with an $n_i$-shift filtration.
\end{prop}

\begin{proof} This is only a slightly generalization of the case $I$ finite; see proof of \cite[Theorem 3.7]{Forre}.
\end{proof}
Then by Theorem \ref{USE} and Proposition \ref{lemm_clef} we firstly get $$P_\G(t)=P_\BB(t)=\big(1-dt +\sum_{i\in I} t^{n_i}\big)^{-1}\cdot$$

Consider now the continuous morphism
$$\Psi : \prod_{i\in I} \Lambda_\G \rightarrow \II(\R)/\II(\R)\Ee_1 \simeq \R/\R^p[\R,\R],$$
sendind $(a_i)$ to $\sum_i a_i Y_i \ ({\rm mod} \ \II(\R)\Ee_1)$; as  $n_i \rightarrow \infty$ with $i$, it is well-defined. Remember that $\Lambda_\G \simeq \Ee/\II(\R)$.
 Put $\N=\ker(\Psi)$. 
 
 \begin{lemm} The map $\Psi$ is surjective.
 \end{lemm}
 
 \begin{proof}  Put $\W=\{ \sum_{i\in I} a_i Y_i, a_i \in \Ee \} \subset \II(\R)$.  Then 
 $$\begin{array}{rcl} \II(\R)&=&\W  \Ee\\
  &=&   \W \F_p  + \Ee \W \Ee_1 = \W  + \W \Ee_1 .
 \end{array}$$
 We conclude by observing that $\W \Ee_1 \subset \II(\R)\Ee_1$.
 \end{proof}

 Therefore one gets  a sequence of filtered $\G$-modules:
 $$1\to \N \to \prod_{i \in I} \Lambda_\G[n_i] \stackrel{\Psi}{\to} \II(\R)/\II(\R)\Ee_1 \to 1.$$ This one
 induces the following sequence of graded $\AA$-modules:
 $$0 \to \grad(\N) \to \grad(\prod_{i \in I} \Lambda_\G[n_i]) \to \grad(\II(\R)/\II(\R)\Ee_1)\to 0. $$
 For the surjectivity, use the fact that  $I$ is countable.
 Now as $n_i \to \infty$ with $i$, then $$\displaystyle{\grad\big(\prod_{i\in I} \Lambda_\G[n_i]\big)=\grad\big(\bigoplus_{i\in I}{\Lambda_\G}[n_i] \big) \simeq \bigoplus_{i\in I}{\BB [n_i]}.}$$
 
By Proposition \ref{lemm_clef}, we finally get that $\Psi$ induces an isomorphism between 
$\displaystyle{\grad\big(\prod_{i \in I} \Lambda_\G[n_i]\big)}$ and $\displaystyle{\grad\big(\II(\R)/\II(\R)\Ee_1\big)}$, which implies  $\grad(\N)=0$, then $\N=0$. Hence, as $\G$-modules,  $\displaystyle{\prod_{i\in I} \Lambda_\G \simeq  \II(\R)/\II(\R)\Ee_1 \simeq \R/\R^p[\R,\R]}$, and we conclude with  Theorem \ref{theo_dimcoh}.
\end{proof}

\begin{rema} \label{remark_stronglyfree}
Conclusions of Theorem \ref{EQ} also hold if  $\{\widehat{\rho_i}\}_{ i \in I}$ is strongly free.
\end{rema}

\begin{rema}
 For references on graded and filtered modules, see also \cite[Chapter I and II]{Lazard}.
\end{rema}

%%%%%%%%%%%%%%%%%%%%%%%%%%%%%%%%%%%%%%%%%%%%%%%%%%%%%%%%%%%%%%%%ùùù
%%%%%%%%%%%%%%%%%%%%%%%%%%%%%%%%%%%%%%%%%%%%%%%%%%%%%%%%%%%%%%%%ùùù

\subsubsection{Cup-products and cohomological dimension}
Here we suppose now $p>2$.

Let $\G$ be a pro-$p$ group of $p$-rank~$d$ which is not pro-$p$ free. Recall that the cup product sends $H^1(\G) \otimes H^1(\G)$ to 
$H^2(\G)$. Labute in \cite{Labute} gave a criterion involving cup-products so that $\cd(\G) =2$. This point of view has been developped by Forré in \cite{Forre}. Let us recall it.

\begin{theo}[Forr\'e] \label{theo_cupproduct} Let $p>2$ be a prime number. Let $\G$ be a finitely presented pro-$p$ group which is not pro-$p$ free.
 Suppose that $H^1(\G)=U\oplus V $ such that $U\cup U =0$ and $U\cup V=H^2(\G)$. Put $c=\dim V$.
 Then $\cd(\G)=2$, and $\G$ can be described by some  relations  $\rho_1,\cdots, \rho_r$ such that the highest term of each  $\rho_i$   can be  written  as $X_{t(i)}X_{s(i)}$ for some $s(i),t(i)$ such that  $s(i)\leq c <t(i)$, and such that $(s(i),t(i))\neq (s(j),t(j))$ for $i\neq j$. 
\end{theo}

\begin{proof}
 See the proof of  \cite[Theorem 6.4, Corollary 6.6]{Forre} with  the choice of the ordering $X_1 < X_2 < \cdots < X_d$. 
\end{proof}

\begin{rema}
 Observe that the family $\{X_{s(i)}X_{t(i)}\}_i$ of Theorem \ref{theo_cupproduct} is combinatorially free.
\end{rema}

Before to present a corollary, let us make the following observation: given $n \geq 1$, thanks to Example \ref{exemple_clef}, one may find some  $x \in \FF$ such that the highest term of $x$ is like $X_kX_j^nX_i$ for $i<j<k$.

\begin{coro} \label{coro_cupproduct}
 Consider the situation of Theorem \ref{theo_cupproduct}.  Suppose $c\geq 2$.
 For some fixed $1<i_0\leq c <j_0 \leq d$, and $n \geq 1$,  let $x_n \in \FF$ of highest term $X_{j_0}X_{i_0}^nX_1$. Suppose moreover  that  $r <  (d-c)(c-1)$.  
 Then there exists $(i_0,j_0)$ such that the family $\{\widehat{\rho_1},\cdots, \widehat{\rho_r}, \widehat{x_n}, n\geq 1\}$ is combinatorially  free. In particular for such $(i_0,j_0)$:
 \begin{enumerate}[\quad (i)]
  \item the group quotient $\Gamma:=\FF/\langle \rho,\cdots, \rho_r, x_n, n \in \Z_{>0}\rangle^{\Nor}$ of $\G$ is of cohomological dimension~$2$;
  \item $\dim H^2(\Gamma,\F_p)=\infty$;
  \item The Poincar\'e series of $\Lambda_\Gamma$  is $\displaystyle{\big({1-dt+rt^2+ t^3\sum_{n\geq 0} t^{n}}\big)^{-1}}$.
 \end{enumerate}

\end{coro}

 \begin{proof}
Thanks to Theorem \ref{theo_cupproduct}, for $i=1,\cdots, r$, the highest term of $\rho_i$ is of the form $X_{t(i)}X_{s(i)}$ for some $s(i)\leq c <t(i)$, and the family $\E:=\{X_{t(1)}X_{s(1)},\cdots, X_{t(r)}X_{s(r)}\}$ is combinatorially free. 
Now, as $r < (d-c)(c-1)$ and $c\geq 2$, we can find $(i_0,j_0)$   such that  $X_{j_0}X_{i_0}$ is not in $\E$, and then $\E\cup\{ X_{j_0}X_{i_0}^nX_1, n\in \Z_{>0}\}$  is  combinatorially free. Then apply Theorem \ref{EQ}.
 \end{proof}

 \begin{rema}
  In fact $r \leq (d-c)c-2$ is sufficient. Indeed, with such condition one has $X_{j_0}X_{i_0}\notin \E$ for some $(i_0,j_0)\neq (1,r)$,   $i_0\leq c <j_0\leq r$.
  Hence, if $i_0\neq 1$ the family $\E\cup \{X_{j_0}X_{i_0}^nX_1, n\in \Z_{>0}\}$ is combinatorially free. Otherwise $j_0\neq r$, and take $\E\cup\{ X_r  X_{j_0}^nX_{i_0}, n\in \Z_{>0}\}$. %Observe that the two Poincaré series are not the same.
 \end{rema}

%%%%%%%%%%%%%%%%%%%%%%%%%%%%%%%%%%%%%%%%%%%%%%%%%%%%%%%%%%%%%%%%ùùù
%%%%%%%%%%%%%%%%%%%%%%%%%%%%%%%%%%%%%%%%%%%%%%%%%%%%%%%%%%%%%%%%ùùù
 
%%%%%%%%%%%%%%%%%%%%%%%%%%%%%%%%%%%%%%%%%%%%%%%%%%%%%%%%%%%%%%%%ùùù

 \subsection{Arithmetic backgrounds} \label{section_arithmetic}
 Let $p$ be a prime number, and let $\K$ be a number field. For $p=2$, assume $\K$ totally imaginary.
 Let $S$ and $T$ two disjoint finite  sets of prime ideals of  the ring of integers $\O_\K$ of $\K$. We assume moreover that each $\p \in S$ is such that $|\O_\K/\p| \equiv 1 (\mod \ p)$; the set $S$ is called tame.
 We denote by $\Cl_\K^T(p)$ the $p$-Sylow of the $T$-class group of $\K$. 
 
 Let $\K_S^T/\K$ be the maximal pro-$p$ extension of $\K$ unramified outside $S$ and where each  $\p \in T$ splits completely in $\K_S/\K$; put $\G_S^T=\Gal(\K_S^T/\K)$. 
 As we  recalled it in Introduction, when $\G_S^T$ is infinite, the extension $\K_S^T/\K$
 is asymptotically good.
 Recall Shafarevich's formula (see for example \cite[Chapter I, \S, Theorem 4.6]{Gras}): $$d_p \G_S^T = |S|-(r_1+r_2)-1-|T|+\delta_{\K,p}+ d_p \V_S^T/\K^{\times p},$$ where $$\V_S^T=\{x\in \K^\times, \ x\in \K_\p^pU_\p \ \forall x\notin S\cup T, \ x \in \K_\p^p \ \forall \p \in S \},$$ and where $\delta_{\K,p}=1$ if $\K$ contains $\mu_p$ (the $p$-roots of $1$), $0$ otherwise. Here as usual, $\K_\p$ is the completion of $\K$ at $\p$,  and $U_\p$  is the group of the local units at $\p$.
 Observe that  if there is no $p$-extension of $\K(\mu_p)$ unramified outside $T$ and $p$ in which each prime of $S$ splits completely, then $\V_S^T/\K^{\times \ p}$ is trivial:  this is a Chebotarev condition type. 
 
 \medskip
 
 Schmidt in \cite{Schmidt} showed that $\G_S^T$ may be {\it mild} following the terminology of Labute \cite{Labute}. More precisely, he proved:
 
 \begin{theo}[Schmidt] \label{theo_schmidt} Let $\K$ be a number field and let $p$ be a prime number. For $p=2$ suppose $\K$ totally imaginary.
  Let $S_0$ and $T$ two disjoint finite sets of prime ideals of $\K$ with $S_0$ tame. Assume $T$ sufficiently large such that $\Cl_\K^T(p)$ is trivial; when $\mu_\p \subset \K$, assume moreover that $T$ contains all prime ideals above $p$.
  Then there exist infinitely finite tame sets $S$ containing $S_0$ such that  $H^1(\G_S^T)=U\oplus V$ where the two subspaces $U$ and $V$ satisfy: $(i)$ $U\cup U=0$; $(ii)$ $U\cup V=H^2(\G_S^T)$. Moreover, for such $S$ and $T$ one has $\dim H^2(\G_S^T)=\dim H^1(\G_S^T)+r_1+r_2+|T|-1$.
 \end{theo}

 Theorem \ref{theo_schmidt} is not presented in this form in \cite{Schmidt}, here  we give the form we need: the result  presented here can be found in the proof of Theorem 6.1 of \cite{Schmidt}. 

 \medskip
 
 At this level, let us compute the value of $c=\dim V$ of Theorem \ref{theo_schmidt}, following \cite{Schmidt}.
 
 %Recall that $T$ is such that $\Cl_K^T(p)$ is trivial.
 When $\mu_p \nsubset \K$ let us choose first  a finite set $S_0$ of prime ideals of $\K$, tame and disjoint from $T$, such that for every $\p \in S_0$, one has $$d_p \G_{S_0\backslash\{\p\}}^T=|S_0|-r_1-r_2-|T|+\delta_{\K,p},$$
 which is equivalent by Shafarevich's formula  to the triviality of $\V_{S_0\backslash \{\p\}}^T/\K^{\times p}$. 
 
 When $\mu_p \subset \K$ let us choose  $S_0$, finite,  tame and disjoint from $T$, such that the set of the Frobenius  at $\p$ in $\G_T^{p-el}$ when $\p$ varies in $S_0$, corresponds to the nontrivial elements of
 $\G_T^{p-el}$, where   $\G_T^{p-el}$ is the Galois group of the $p$-elementary abelian extension $\K_T^{p-el}/\K$ of $\K_T/\K$. Here one has also the triviality of $\V_{S_0\backslash \{\p\}}^T/\K^{\times p}$.
 
 The set $S$ of Theorem \ref{theo_schmidt} contains $S_0$, and is of size $2|S_0|$; the prime ideals $\p \in S-S_0$ are choosen by respecting some global conditions,   thanks to Chebotarev density theorem. Moreover  $U=H^1(\G_{S_0}^T,\F_p)$, and the subspace  $V$ is such that $\dim V=c=|S_0|$. See \cite[Proof of Theorem 6.1]{Schmidt} for more details.
 
 Now observe the following:
 
 \begin{lemm} \label{lemm_p-ram}
  Above the previous conditions, each prime $\p\in S$ is ramified in the $p$-elementary abelian extension $\K_S^{T, p-el}/\K$ of $\K_S^{T}/\K$.
 \end{lemm}

 \begin{proof} Observe first that if $S'' \subset S'$, then $\V_{S'}^T/\K^{\times p} \hookrightarrow \V_{S''}^T/\K^{\times p}$.
 
 Hence thanks to  the choice of $S_0$, it is not difficult to see  the following: for every $\p\in S$, $\V_{S\backslash \{\p\}}^T/\K^{\times p}$ is trivial. Then by Shafarevich's formula, we get that 
 $$d_p \G_S^T=1+ d_p \G_{S\backslash \{\p\}}^T,$$
 showing that $\p$ is ramified in $\K_S^{T, p-el}/\K$.
 \end{proof}

 \medskip
 
 Put $\alpha_{\K,T}=3+2\sqrt{2+r_1+r_2+|T|}$. In Theorem \ref{theo_schmidt} one may take  $S$ sufficiently large so that $d=\dim H^1(\G_S^T,\F_p)>  \alpha_{\K,T}$.
 
 \begin{lemm} \label{lemma_easy}
  If $d > \alpha_{\K,T}$, then $d+r_1+r_2+|T|-1<(d-c)(c-1)$ for every $c\in [2,d]$.
 \end{lemm}

 \begin{proof}
  Easy computation.
 \end{proof}

 Let us finish this part with an obvious observation thanks to class field theory.
 
 \begin{rema}
  If $\G_S^T$ is not trivial and of cohomological dimension at most $2$, then $\cd(\G_S^T)=2$.
 \end{rema}

%%%%%%%%%%%%%%%%%%%%%%%%%%%%%%%%%%%%%%%%%%%%%%%%%%%%%%%%%%%%%%%%ùùù

%%%%%%%%%%%%%%%%%%%%%%%%%%%%%%%%%%%%%%%%%%%%%%%%%%%%%%%%%%%%%%%%ùùù
%%%%%%%%%%%%%%%%%%%%%%%%%%%%%%%%%%%%%%%%%%%%%%%%%%%%%%%%%%%%%%%%ùùù

 \section{Example and proof}

\subsection{Example}

$\bullet$ Take $p>2$, and $\K=\Q$. In this case the relations of the pro-$p$ groups $\G_S$ are all local, and then not difficult to describe: this is the description due to Koch \cite[Chapter 11, \S 11.4,  Example 11.11]{Koch}. 
 
 \medskip
 
 Let $\ell$ be a prime number such that $p| \ell-1$. Denote by $\Q_\ell$ the (unique) cyclic degree $p$-extension of $\Q$ unramified
outside $\ell$; the extension $\Q_\ell/\Q$ is totally ramified at $\ell$.

\medskip

Let $S=\{\ell_1,\cdots, \ell_d\}$ be $d$ different prime numbers such that $p$ divides each $\ell_i-1$.
The pro-$p$ group $\G_S$ can be described by $x_1,\cdots, x_d$ generators, and $\rho_1,\cdots,\rho_d$ relations verifying:
\begin{eqnarray} \label{relations} \rho_i=\prod_{j\neq i}[x_i,x_j]^{a_j(i)} \mod \ \FF_3,\end{eqnarray}
where $a_j(i) \in \Z/p\Z$; moreover the element $x_i$ can be chosen such that it is  a generator of the inertia group of $\ell_i$.
The element $a_j(i)$ is zero if and only the prime $\ell_i$ splits in $\Q_{\ell_j}/\Q$, which is equivalent to $$\ell_i^{(\ell_j-1)/p}\equiv 1 (\mod \ \ell_j).$$

\

 $\bullet$ Typically take $p=3$, and $S_0=\{7,13\}$, $T=\emptyset$. Then put $S=\{p_1,p_2,p_3,p_4,p_5\}$
 with $p_1=31, p_2=19, p_3=13, p_4=337, p_5=7$.
 Then the highest terms of the relations~(\ref{relations}), viewed in $\F_p^{nc}\lbrack X_1,\cdots, X_5 \rbrack$, are $\widehat{\rho_1}=X_1X_3$, $\widehat{\rho_2}=X_2X_4$, $\widehat{\rho_3}=X_2X_3$, $\widehat{\rho_4}=X_1X_4$, $\widehat{\rho_5}=X_1X_5$.
 Hence as the $\widehat{\rho_i}$'s are combinatorially free, then $\G_S$ is of cohomological dimension~$2$ by Theorem~\ref{EQ}.
 
 \medskip
 
 Now for each $n\in \Z_{>0}$, let us choose a prime number $p_n$ of $\Z$ such that the highest term in $\F_p^{nc}\ldbrack X_1,\cdots, X_5\rdbrack$  of its Frobenius $\sigma_n \in \G_S$  is like $X_5X_4^nX_1$ (which is possible by Example \ref{exemple_clef} or Corollary  \ref{coro_cupproduct},  see  next section). 
 Then consider the maximal Galois subextension $\L/\Q$ of $\Q_S/\Q$ fixed by all the conjuguates of the $\tau_n$'s (this is the ``cutting towers'' strategy of \cite{HMR}).  Put $\G=\Gal(\L/\Q)$. Then the pro-$3$ group $\G$ can be described by the generators $x_1,\cdots, x_5$, and the relations $\{\rho_1,\cdots, \rho_5, \tau_n, n\in \Z_{>0}\}$ (which is not {\it a priori} a minimal set). By construction all the $p_n$ split totally in $\L/\Q$.
 Observe now that $$\{\widehat{\rho_1},\cdots, \widehat{\rho_5}, \widehat{\tau_n}, n\geq 1\}=\{X_5X_1,X_5X_2,X_4X_3,X_4X_2,X_5X_3,X_5X_4^nX_1, n\in \Z_{>0}\},$$ which is combinatorially free. By Theorem \ref{EQ} the pro-$3$-group $\G$ is of cohomological dimension $2$, $H^2(\G)$ is infinite, and  $\F_3\ldbrack \G \rdbrack$ has $\displaystyle{\big({1-5t+5t^2+t^3(1+t+t^2+\cdots) }\big)^{-1}}$ as Poincar\'e series.

 \subsection{Proof of the main result} \label{section_proof}

 $\bullet$ Let $p>2$ be a prime number, and let $\K$ be a number field. 
 Let $S_0$ and $T$ two finite disjoint sets of prime ideals of $\K$, where $S_0$ is tame. Take $T$ sufficiently large such that $\Cl_\K^T(p)$ is trivial. When $\K$ contains $\mu_p$, assume moreover that $T$ contains all $p$-adic prime ideals.
 
 \medskip
 
 First take $S$ containing $S_0$ as in Theorem \ref{theo_schmidt}, and sufficiently large such that $d > \alpha_{\K,T}$. Put $\G=\G_S^T$. Here $r=\dim H^2(\G)=d+r_1+r_2-1+|T|$.
 
\medskip

Let us start with a minimal presentation of $\G$:
$$1\longrightarrow \R \longrightarrow \FF \stackrel{\varphi}{\longrightarrow} \G \longrightarrow 1.$$
By Theorem \ref{theo_schmidt} and Theorem  \ref{theo_cupproduct} the quotient $\R/\R^p[\FF,\R]$ may be generated as $\F_p$-vector spaces by some relations $ \rho_1,\cdots, \rho_r$ such that the highest terms $\widehat{\rho_k}$   are like  $X_iX_j$ for some  $i\leq c <j$, where $c=\dim V$. Observe that as $\G$ is FAb then $c \in  [2,d-2]$.

\medskip

Given $n\geq 1$, then the quotient $\G/\G_{n+1}$ is finite. Put $\K_{(n+1)}=(\K_S^T)^{\G_{n+1}}$.
Let $a_n \in \Z_{>0}$ be an increasing sequence.
Let $x_n\in \FF_{a_n} \backslash \FF_{a_n+1}$.
By Chebotarev density theorem there exists some prime ideal $\p_n \subset \O_\K$ such that $\sigma_{\p_n}$ is conjuguate to $x_n$ in $\Gal(\K_{(a_n+1)}/\K)$. Here $\sigma_{\p_n} \in \G$ denotes the Frobenius of $\p_n$ in $\K_S^T/\K$. Now take $z_n\in \FF$ such that $\varphi(z_n)=\sigma_{\p_n}$. 
Hence $$z_n \equiv \sigma_{\p_n} \ (\mod \ \R\FF_{a_n+1}).$$
In other words, there exists $y_n \in \FF_{a_n+1}$ and $r_n\in \R$ such that  $z_n=\sigma_{\p_n} y_n r_n$.

\medskip

Let $\Sigma=T\cup \{\p_1, \p_2,\cdots\}$, and consider $\K_S^\Sigma$ the maximal pro-$p$ extension of $\K$ unramified oustide $S$  and where each primes $\p_i$ of $\Sigma$ splits completely. Put $\G_S^\Sigma=\Gal(\K_S^\Sigma/\K)$. 
Then $$\G_S^\Sigma\simeq \G/\langle \sigma_{\p_n}, n \in \Z_{>0} \rangle^{\Nor}.$$ 
Here $\langle \sigma_{\p_n}, n \in \Z_{>0} \rangle^{\Nor}$ is the normal closure of $\langle \sigma_{\p_n}, n \in \Z_{>0} \rangle$ in $\G_S^\Sigma$. Hence $\K_S^\Sigma$ satisfies $(i)$ and $(ii)$ of Theorem \ref{theo_main}.
But observe now that
$$\G/\langle \sigma_{\p_n}, n \in \Z_{>0} \rangle^{\Nor} \simeq \FF/\langle  \rho_1,\cdots, \rho_r, z_n,  n\in \Z_{>0}\rangle^{\Nor} = \FF/\langle  \rho_1,\cdots, \rho_r, x_ny_n,  n\in \Z_{>0}\rangle^{\Nor},$$
as $\sigma_{\p_n}$ and $x_n$ are conjugate.
\medskip

Since the highest term of each  $x_ny_n $ in $\Ee=\F_p^{nc}\ldbrack X_1,\cdots, X_d\rdbrack$ is the same as the highest term of $x_n$, it suffices to choose the $x_n$'s as in Corollary \ref{coro_cupproduct} which is possible: indeed  as $d > \alpha_{\K,T}$ then by Lemma \ref{lemma_easy} $r<(c-1)(d-c)$,  for every $c \in [1,d-1]$.
Thanks to Corollary \ref{coro_cupproduct}, one gets $(iii)$, $(iv)$, and $(vi)$ of Theorem \ref{theo_main}.

\medskip

$(v)$: by Lemma \ref{lemm_p-ram}  each prime ideal $\p\in S$ is ramified in $\K_S^{T, p-el}/\K$, showing  that  $\tau_\p \in \G$ is not in $\R\FF^p[\FF,\FF]$, where $\tau_\p$ is  a generator of the inertia group at $\p$ in $\G$.  As the $p$-rank of $\G_S^\Sigma$ is the same as the $p$-rank of $\G$, each prime $\p \in S$ is ramified in $\K_S^\Sigma$. But as $\G$ is without torsion (because $\cd(\G)=2$), necessarly $\langle \tau_\p \rangle \simeq \Z_p$, and the structure of local extensions forces $(\K_S^\Sigma)_\p/\K_\p$ to be maximal.

\medskip

$\bullet$  Assume $p=2$, and $\K$ be totally imaginary. Then Theorem \ref{theo_schmidt} holds, but Theorem \ref{theo_cupproduct} does not. As explained by Forr\'e in \cite[Proof Theorem 6.4]{Forre}, one has to take two orderings to show that the highest terms of the relations $\rho_1,\cdots, \rho_r$ are strongly free. Now in this context  the strategy of the approximation of elements $x_n$ by some Frobenius  as in Corollary \ref{coro_cupproduct} also applies. Then by following the proof of Theorem 6.4 in \cite{Forre}, and by choosing the $x_n$'s as in the case $p\neq 2$, we  observe that the initial forms of the new relations $\{\rho_1,\cdots, \rho_r, x_n, n\geq 1\}$ are still strongly free. We conclude by using Remark \ref{remark_stronglyfree} of  Theorem \ref{EQ}.

\qed

%%%%%%%%%%%%%%%%%%%%%%%%%%%%%%%%%%%%%%%%%%%%%%%%%%%%%%%%%%%%%%%%%%%%%%
%%%%%%%%%%%%%%%%%%%%%%%%%%%%%%%%%%%%%%%%%%%%%%%%%%%%%%%%%%%%%%%%%%%%%%ùù

\end{document}